\documentclass[12pt]{amsart}

\usepackage{amssymb,amsfonts, amsthm,amsmath}
\usepackage[utf8]{inputenc} 
\usepackage{fullpage}
\usepackage{enumerate}
\usepackage[pdfpagelabels]{hyperref}
\usepackage{xcolor}

\newtheorem{theorem}{Theorem}
\newtheorem{lemma}{Lemma}[section]

\newtheorem{proposition}[lemma]{Proposition}

\newtheorem{note}[lemma]{Note}
\theoremstyle{definition}
\newtheorem{definition}[lemma]{Definition}
\newtheorem{example}[lemma]{Example}

\DeclareMathOperator{\inter}{int}
\DeclareMathOperator{\bd}{bd}
\DeclareMathOperator{\conv}{conv}
\newcommand{\x}{\mathbf{x}}
\newcommand{\y}{\mathbf{y}}
\newcommand{\bb}{\mathbf{b}}
\newcommand{\Ze}{\mathbb{Z}}
\newcommand{\ve}{\mathbf{v}}
\newcommand{\ev}{\mathbf{e}}
\newcommand{\oo}{\mathbf{o}}
\newcommand{\K}{\mathbf{K}}
\newcommand{\SSS}{\mathbf{S}}
\newcommand{\LL}{\mathbf{L}}
\newcommand{\PPP}{\mathbf{P}}
\newcommand{\E}{\mathbb{E}}
\newcommand{\B}{\mathbf{B}}
\newcommand{\PP}{\mathcal{P}}

\newcommand{\csep}[1][\K,n]{c_{\textnormal{sep}}\left(#1\right)}
\newcommand{\lsep}[1][\K]{\lambda_{\textnormal{sep}}\left(#1\right)}
\newcommand{\dsep}[1][\rho,\K]{\delta_{\textnormal{sep}}\left(#1\right)}
\newcommand{\Iq}[1]{\IIqq(#1)}
\DeclareMathOperator{\IIqq}{Iq}
\newcommand{\floor}[1]{\left\lfloor#1\right\rfloor}

\renewcommand{\Re}{{\mathbb R}}
\newcommand{\Red}{\Re^d}
\newcommand{\Ed}{\E^{d}}
\newcommand{\st}{\; : \; }
\newcommand{\spann}{\operatorname{span}}
\newcommand{\vol}[2][d]{\operatorname{vol}_{#1}\left(#2\right)}
\newcommand{\iprod}[2]{\left<#1,\linebreak[0]#2\right>}

\newcommand{\xn}{\ensuremath{\x_1,\linebreak[0]\ldots,\linebreak[0]\x_n}}
\newcommand{\fn}{\ensuremath{f_1,\linebreak[0]\ldots,\linebreak[0]f_n}}
\newcommand{\xfn}{\ensuremath{(\x_1,f_1),\linebreak[0]\ldots,\linebreak[0](\x_n,
f_n)}}
\newcommand{\pairs}{vector-linear functional pairs}
\newcommand{\ijn}{\ensuremath{1\leq i,j\leq n, i\neq j}}
\newcommand{\iin}{\ensuremath{1\leq i\leq n}}

\newcommand{\Hsep}{H_{\rm sep}}
\newcommand{\hsep}{h_{\rm sep}}
\newcommand{\difour}{d\in\{1,2,3,4\}}

\title[Contact graphs of totally separable packings]{On contact graphs of 
totally separable packings in low dimensions
}


\author[K. Bezdek \and M. Nasz\'odi]{K\'{a}roly Bezdek \and M\'arton Nasz\'odi}
\address[K.B.]{Department of Mathematics and Statistics, University of Calgary, 
Canada.}
\address{Department of Mathematics, University of Pannonia, Veszpr\'em, 
Hungary.}
\email{bezdek@math.ucalgary.ca}

\address[M.N.]{Department of Geometry, E\"otv\"os Lor\'and University, 
Budapest, 
Hungary}
\email{marton.naszodi@math.elte.hu}

\begin{document}
\date{}
\begin{abstract}
The \emph{contact graph} of a packing of translates of a convex body in 
Euclidean $d$-space $\mathbb E^d$ is the simple graph whose vertices are the 
members of the packing, and whose two vertices are connected by an edge if the 
two members touch each other. A packing of translates of a convex body is 
called \emph{totally separable}, if any two members can be separated by a 
hyperplane in $\mathbb E^d$ disjoint from the interior of every packing element.

We give upper bounds on the maximum vertex degree (called \emph{separable Hadwiger 
number}) and the maximum number of edges (called \emph{separable contact 
number}) of the contact graph of a totally separable packing of $n$ translates 
of an arbitrary smooth convex body in $\mathbb E^d$ with $d=2,3,4$. In the proofs, 
linear algebraic and convexity methods are combined 
with volumetric and packing density estimates based on the underlying 
isoperimetric (resp., reverse isoperimetric) inequality.
\end{abstract}
\maketitle


\sloppy

\renewcommand\footnotemark{} 
\section{Introduction}\label{sec:intro}

We denote the $d$-dimensional Euclidean space by $\Ed$, and the 
unit ball centered at the origin $\oo$ by $\B^d$. A \emph{convex 
body} $\K$ is a compact convex subset of $\Ed$ with nonempty interior. 
Throughout the paper, $\K$ always denotes a convex body in $\Ed$.
If $\K = -\K:=\{-x: x\in \K\}$, then $\K$ is said to be \emph{$\oo$-symmetric}. 
$\K$ is said to be \emph{smooth} if at every point on the boundary $\bd \K$ of 
$\K$, the body $\K$ is supported by a unique hyperplane of $\Ed$. $\K$ is 
\emph{strictly convex} if the boundary of $\K$ contains no nontrivial line 
segment.

The \emph{kissing number problem} asks for the maximum number $k(d)$ of 
non-overlapping translates of $\B^d$ that can touch $\B^d$. Clearly, 
$k(2)=6$. To date, the only known kissing number values are $k(3)=12$ 
\cite{ScWa53},
$k(4)=24$ \cite{Mu08}, $k(8)=240$ \cite{OdSl79}, and 
$k(24)=196560$ \cite{OdSl79}. 
For a survey of kissing numbers we refer the interested reader to 
\cite{Bo}. 

Generalizing the kissing number, the \emph{Hadwiger number} or \emph{the 
trans\-la\-ti\-ve kissing number} $H(\K)$ of a convex body 
$\K$ is the maximum number of non-overlapping translates of $\K$ that all touch 
$\K$. Given the difficulty of the kissing number problem, determining 
Hadwiger numbers is highly nontrivial with few exact values known for $d\ge 3$. 
The best general upper and lower bounds on $H(\K)$ are due to Hadwiger 
\cite{Ha} and Talata \cite{Ta} respectively, and can be expressed as 
\begin{equation}\label{eq:hadwiger}
2^{cd}\le H(\K)\le 3^{d} - 1, 
\end{equation}
where $c$ is an absolute constant and equality holds in the right inequality if 
and only if $\K$ is an affine $d$-dimensional cube \cite{Gr61}.

A packing of translates of a \emph{convex domain}, that is, a convex 
body $\K$ in $\E^2$ is said to be \emph{totally separable} if any two packing 
elements can be separated by a line of 
$\E^{2}$ disjoint from the interior of every packing element. This 
notion was introduced by G. Fejes T\'{o}th and L. Fejes T\'{o}th \cite{FTFT73}.

We can define a totally separable packing of translates of a 
$d$-di\-men\-si\-o\-nal 
convex body $\K$ in a similar way by requiring any two packing elements to be 
separated by a hyperplane in $\Ed$ disjoint from the interior of 
every packing element \cite{BeSzSz, Ke}.

Recall that the \emph{contact graph} of a packing of translates of $\K$ is the 
simple graph whose vertices are the members of the packing, and whose two 
vertices are connected by an edge if and only if the 
two members touch each other. In this paper we 
investigate the maximum vertex degree (called \emph{separable Hadwiger number}), as 
well as the maximum number of edges (called the \emph{maximum separable contact 
number}) of the contact graphs of totally separable packings by a 
given number of translates of a smooth or strictly convex body $\K$ in $\Ed$. 
This extends and generalizes the results of \cite{BeKhOl} and \cite{BeSzSz}.  
The details follow. 

\subsection{Separable Hadwiger numbers}\label{sec:sepHintro}
It is natural to introduce the totally separable analogue of the 
Hadwiger number as follows \cite{BeKhOl}.

\begin{definition}
Let $\K$ be a convex body in $\Ed$. We call a family of translates of $\K$ that 
all touch 
$\K$ and, together with $\K$, form a totally separable packing in $\Ed$ a 
\emph{separable Hadwiger configuration} of $\K$.
The \emph{separable Hadwiger number} 
$\Hsep(\K)$ of $\K$ is the maximum size of a separable Hadwiger configuration 
of $\K$.
\end{definition} 

Recall that the \emph{Minkowski symmetrization} of the convex body 
$\K$ in $\Ed$ denoted by $\K_{\oo}$ is defined by
$\K_{\oo} := \frac{1}{2}(\K + (- \K)) = \frac{1}{2}(\K - \K) = \frac{1}{2}\{\x 
- \y : \x , \y \in \K\}$. Clearly, $\K_{\oo}$ is an $\oo$-symmetric 
$d$-dimensional convex body.
Minkowski \cite{Mi04} showed that if ${\mathcal 
P}=\{\x_1+\K, \x_2+\K, \dots ,\x_n+\K\}$
is a packing of translates of $\K$, then ${\mathcal P}_{\oo}=\{\x_1+\K_{\oo}, 
\x_2+\K_{\oo}, \dots ,\x_n+\K_{\oo}\}$
is a packing as well. Moreover, the contact graphs of ${\mathcal P}$ and 
${\mathcal P}_{\oo}$ are the same. Using
the same method, it is easy to see that Minkowski's above statement applies to 
totally separable packings as well.
(See also \cite{BeKhOl}.)
Thus, from this point on, we only consider $\oo$-symmetric convex bodies.

It is mentioned in \cite{BeSzSz} that based on \cite{DH} (see also, \cite{Ra} 
and \cite{Ku}) it follows in a straightforward way that $\Hsep(\B^d)=2d$ 
for all $d\geq 2$. On the other hand,
if $\K$ is an $\oo$-symmetric convex body in $\Ed$, then each facet
of the minimum volume circumscribed parallelotope of $\K$ touches $\K$ at the
center of the facet and so, clearly $\Hsep(\K)\geq 2d$. Thus, 
\begin{equation}\label{basic-bounds}
2d\leq \Hsep(\K)\leq H(\K)\leq 3^d-1
\end{equation}
holds for any $\oo$-symmetric convex body $\K$ in $\Ed$. 
Furthermore, the $d$-cube is the only $\oo$-symmetric convex body
in $\E^d$ with separable Hadwiger number $3^d-1$ \cite{Gr61}.

We investigate equality in the first inequality of \eqref{basic-bounds}.
First, we note as an easy exercise that $\Hsep$ as a map from 
the set of convex bodies equipped with any reasonable topology to the reals is 
upper semi-continuous. Thus, for any $d$, if an $\oo$-symmetric convex body 
$\K$ in $\Ed$ is sufficiently close to the Euclidean ball $\B^d$ (say, 
$\B^d\subseteq \K\subseteq (1+\varepsilon_d)\B^d$, where $\varepsilon_d>0$ 
depends on $d$ only), then $\Hsep(\K)=2d$.

Hence, it is natural to ask whether the set of those $\oo$-symmetric convex 
bodies in $\Red$ with $\Hsep(\K)=2d$ is dense. In this paper, we investigate 
whether $\Hsep(\K)=2d$ holds for any $\oo$-symmetric smooth or strictly convex 
$\K$ in $\Ed$. Our first main result is a partial answer to this question.

\begin{definition}
An \emph{Auerbach basis} of an $\oo$-symmetric convex body 
$\K$ in $\Ed$ is a set of $d$ points on the boundary of $\K$ that form a basis 
of $\Ed$ with the property that the hyperplane through any one of them, 
parallel to the other $d-1$ supports $\K$.
\end{definition}

\begin{theorem}\label{thm:smoothstrictlycvx}
Let $\K$ be an $\oo$-symmetric convex body in $\Ed$, which is smooth \emph{or} 
strictly convex. Then 
\begin{enumerate}[(a)]
\item\label{item:threefourdim}
For $\difour$, we have $\Hsep(\K)=2d$ and, in any separable Hadwiger 
configuration of $\K$ with $2d$ translates of $\K$, the translation vectors are 
$d$ pairs of opposite vectors, where picking one from each pair yields an 
Auerbach basis of $\K$.
\item\label{item:DG}
$\Hsep(\K)\leq 2^{d+1}-3$ for all $d\geq 5$.
\end{enumerate}
\end{theorem}

We note that part (a) of Theorem~\ref{thm:smoothstrictlycvx} was 
proved for $d=2$ and smooth $\oo$-symmetric convex domains in \cite{BeKhOl}. We 
prove Theorem~\ref{thm:smoothstrictlycvx} in Section~\ref{sec:thmproofs}.

\subsection{One-sided separable Hadwiger numbers}\label{sec:onesided}

The one-sided Had\-wi\-ger number $h(\K)$ of an $\oo$-symmetric convex body 
$\K$ in 
$\E^d$ has been defined in \cite{BeBr} as the maximum number of non-overlapping 
translates of $\K$ that can touch $\K$ and lie in a {\it closed} supporting 
half-space of $\K$. It is proved in \cite{BeBr} that $h(\K)\leq 2\cdot 
3^{d-1}-1$ holds
for any $\oo$-symmetric convex body $\K$ in $\E^d$ with equality for affine 
$d$-cubes only.


One could consider the obvious extension of the one-side Had\-wi\-ger number to 
separable 
Hadwiger configurations. However, a more restrictive and slightly more 
technical definition serves our purposes better, the reason of which will 
become clear in Theorem~\ref{thm:strict} and Example~\ref{ex:fivedonesidedex}.

\begin{definition}
Let $\K$ be a smooth $\oo$-symmetric convex body in $\Ed$. The \emph{one-sided 
separable Hadwiger number} $\hsep(\K)$ of $\K$ is the maximum number $n$ of 
translates $2\x_1+\K,\ldots,\linebreak[0] 2\x_n+\K$ of $\K$ that 
form a separable Hadwiger configuration of $\K$, and the 
following holds. If \fn{} denote supporting linear functionals of $\K$ at the 
points \xn, respectively, then $\oo\notin\conv\{\xn\}$ and 
$\oo\notin\conv\{\fn\}$.
\end{definition} 

\begin{definition}
For a positive integer $d$, let

\begin{center}
$
 \hsep(d):=\max\{\hsep(\K)\st \K $ is an  
$\oo$-symmetric, smooth and strictly convex body in $\Ed\}$, 
\\
$
 \Hsep(d):=\max\{\Hsep(\K)\st \K $ is an  
$\oo$-symmetric, smooth and strictly convex body in $\Ed\}$, 
\end{center}
and set $\Hsep(0)=\hsep(0)=0$.
\end{definition}

The proof of part (\ref{item:threefourdim}) of 
Theorem~\ref{thm:smoothstrictlycvx} relies on the following fact: for the 
smallest dimensional example $\K$ of an $\oo$-symmetric, smooth and strictly 
convex body with $\Hsep(\K)>2d$, we have $\hsep(\K)>2d$. More precisely,

\begin{theorem}\label{thm:strict}\leavevmode
\begin{enumerate}[(a)]
\item\label{item:twosided}
$\hsep(d)\leq \Hsep(d)\leq\max\{2\ell+\hsep(d-\ell)\st \ell=0,\ldots,d\}$.
\item\label{item:onesided}
$\hsep(d)=d$ for $\difour$.
\item\label{item:onesidedeuclidean}
$\hsep(\B^d)=d$ for the $d$-dimensional Euclidean ball $\B^d$ with $d\in\Ze^+$.
\end{enumerate}
\end{theorem}

According to Note~\ref{note:projection}, when bounding $\Hsep(\K)$ for a smooth 
\emph{or} strictly convex body $\K$, it is sufficient to consider smooth 
\emph{and} strictly convex bodies.

As a warning sign, in Example~\ref{ex:fivedonesidedex} we show that there is an 
$\oo$-symmetric, smooth and strictly convex body $\K$ in $\E^5$, which has a 
set of 6 translates that form a separable Hadwiger configuration, and 
the origin is not in the convex hull of the translation vectors.

We prove Theorem~\ref{thm:strict}, and present Example~\ref{ex:fivedonesidedex} 
in Section~\ref{sec:thmproofs}.

\subsection{Maximum separable contact numbers}
Let $\K$ be an $\oo$-sym\-met\-ric convex body in $\Ed$, and let 
${\mathcal P}:=\{\x_1+\K,\ldots,\x_n+\K\}$ be a packing of translates 
of $\K$. The number of edges in the contact graph of $\mathcal P$ is called the 
\emph{contact number} of $\mathcal P$. Finally let $c(\K,n)$ denote the 
\emph{largest 
contact number} of a packing of $n$ translates of $\K$ in $\E^d$. It is proved 
in \cite{Be02}
that $c(\K,n)\leq \frac{H(\K)}{2}n-n^{\frac{d-1}{d}}g(\K)$ holds for all $n>1$, 
where
$g(\K)>0$ depends on $\K$ only.

\begin{definition}
If $d, n\in\Ze^+$ and $\K$ is an $\oo$-sym\-met\-ric convex body in $\Ed$,
then let  $\csep$ denote the largest 
contact number of a totally separable packing of $n$ translates of $\K$.
\end{definition}

According to Theorem~\ref{thm:smoothstrictlycvx}, the maximum degree in the 
contact graph of a totally separable packing of a smooth convex body $\K$ is 
$2d$, and hence, $\csep\leq dn$, for $\difour$. Our second 
main result is a stronger bound.

\begin{theorem}\label{thm:contactno}
 Let $\K$ be a smooth $\oo$-symmetric convex body in $\Ed$ with $\difour$. 
Then
 \begin{equation*}
  \csep\leq dn-n^{(d-1)/d}f(\K)
  \end{equation*}
 for all $n>1$, where $f(\K)>0$ depends on $\K$ only.
 
 In particular, if $\K$ is a smooth $\oo$-symmetric convex domain in $\E^2$, 
then
 \begin{equation*}
  \csep\leq 2n-\frac{\sqrt{\pi}}{8}\sqrt{n}
  \end{equation*}
 holds for all $n>1$.
\end{theorem}

In \cite{BeKhOl} it is proved that $\csep =\floor{2n-2\sqrt{n}}$ holds for any 
$\oo$-symmetric smooth {\it strictly convex} domain $\K$ and any $n>1$. Thus, 
one may wonder whether the same statement holds for any smooth $\oo$-symmetric 
convex domain $\K$.

We prove Theorem~\ref{thm:contactno} in 
Section~\ref{sec:contactproof}. For a more explicit form
of Theorem~\ref{thm:contactno} see Theorem~\ref{thm:contactnoPrecise} in 
Section~\ref{sec:contactproof}.

\subsection{Organization of the paper}

In Section~\ref{sec:basics} we develop a dictionary that helps translate the 
study of separable Hadwiger configurations of smooth or strictly 
convex bodies to the language of systems of vector--linear functional pairs. In 
Section~\ref{sec:thmproofs}, based on our observations in 
Section~\ref{sec:basics}, we prove Theorem~\ref{thm:strict}, and show how our 
first main result, Theorem~\ref{thm:smoothstrictlycvx} follows from it. 

In Section~\ref{sec:contactproof} we prove our second main result, 
Theorem~\ref{thm:contactno}. This proof is an adaptation of the proof of the 
main result of \cite{Be02} to the setting of totally separable packings of 
smooth convex bodies. One of the main challenges of the adaptation is to 
compute the maximum vertex degree of the contact graph of a totally separable family 
of translates of a smooth convex body $\K$, and to characterize locally the 
geometric setting where this maximum is attained. This local characterization 
is provided by Theorem~\ref{thm:smoothstrictlycvx}.

Finally, in Section~\ref{sec:remarks}, we describe open problems and outline  
the difficulties in translating Theorem~\ref{thm:contactno} to strictly convex 
(but, possibly not smooth) convex bodies.

\section{Linearization, fundamental properties}\label{sec:basics}

First, in order to give a linearization of the problem, we 
consider a set of $n$ pairs \xfn{} where $\x_i\in\Ed$ and $f_i$ is a linear 
functional on $\Ed$ for all $\iin$, and we define the following conditions 
that they may satisfy.

\begin{align}
 f_i(\x_i)=1\;\;\mbox{ and }\;\; f_i(\x_j)\in[-1,0] &\mbox{ holds for all } 
\ijn.
   \label{eq:linearizationFirst}\tag{Lin}\\ 
 f_i(\x_j)=-1, \mbox{ if and only if, } \x_j=-\x_i &\mbox{ holds for all } \ijn.
   \label{eq:linearizationSC}\tag{StrictC}\\
 f_i(\x_j)=-1, \mbox{ if and only if, } f_j=-f_i &\mbox{ holds for all } \ijn.
   \label{eq:linearizationSM}\tag{Smooth}\\
 f_i(\x_i)=1\;\;\mbox{ and }\;\; f_i(\x_j)\in(-1,0] &\mbox{ holds for all } 
\ijn.
   \label{eq:linearization}\tag{OpenLin}
\end{align}

\begin{lemma}\label{lem:linearization}
There is an $\oo$-symmetric, strictly convex body $\K$ in $\Ed$ with 
$\Hsep(\K)\geq n$ if and 
only if, there is a set of $n$ \pairs{} \xfn{} in $\Ed$
satisfying \eqref{eq:linearizationFirst} and \eqref{eq:linearizationSC}.

Similarly, there is an $\oo$-symmetric, smooth convex body $\K$ in $\Ed$ with 
$\Hsep(\K)\geq n$ 
if and only if, there is a set of 
$n$ \pairs{} \xfn{} in $\Ed$
satisfying \eqref{eq:linearizationFirst} and \eqref{eq:linearizationSM}.

Furthermore, the existence of an $\oo$-symmetric, smooth and strictly convex 
body with 
$\Hsep(\K)\geq n$ is equivalent to the existence of 
$n$ \pairs{} satisfying \eqref{eq:linearizationFirst}, 
\eqref{eq:linearizationSC} and \eqref{eq:linearizationSM}.
\end{lemma}

\begin{proof}[Proof of Lemma~\ref{lem:linearization}]
Let $\K$ be an $\oo$-symmetric convex body in $\E^d$.
Assume that $2\x_1+\K,2\x_2+\K,\ldots,2\x_n+\K$ is a separable 
Hadwiger configuration of $\K$, where $\x_1,\dots ,\x_n\in\bd \K$. 
For $\iin$, let $f_i$ be the linear functional corresponding to the separating 
hyperplane of $\K$ and $2\x_i+\K$ which is disjoint from the interior of all 
members of the family. That is, $f_i(\x_i)=1$ and $-1\leq f_i|_K\leq 1$.

Total separability yields that $f_i(\x_j)\in[-1,1]\setminus(0,1)$, for any 
$\ijn$.
Suppose that $f_i(\x_j)=1$. Then $2\x_i+\K$ and $2\x_j+\K$ both touch the 
hyperplane $H:=\{x\in\Ed\st f_i(x)=1\}$ from one side, while $\K$ is on the 
other side of this hyperplane. 

If $\K$ is strictly convex, then this is clearly not possible.

If $\K$ is smooth, then let $S$ be a separating hyperplane of $2\x_i+\K$ 
and $2\x_j+\K$ which is disjoint from $\inter \K$. Since $\K$ is smooth, 
$\K\cap H\cap S=\emptyset$, and hence, $\K$ does not touch 
$2\x_i+\K$ or $2\x_j+\K$, a contradiction.

Thus, if $\K$ is strictly convex or smooth, then \eqref{eq:linearizationFirst} 
holds.

If $\K$ is strictly convex (resp., smooth), then \eqref{eq:linearizationSC} 
(resp., \linebreak[0]\eqref{eq:linearizationSM}) follows immediately.

Next, assume that \xfn{} is a set of $n$ 
\pairs{} satisfying \eqref{eq:linearizationFirst} and 
\eqref{eq:linearizationSC}. We need to show that there is a strictly convex 
body $\K$ with $\Hsep(\K)\geq n$.
Consider the $\oo$-symmetric convex set $\LL:=\{\x\in\Ed\st f_i(\x)\in[-1,1] 
\text{ for all }\iin\}$, the intersection of $n$ $\oo$-symmetric slabs. 

Fix an \iin. If there is no $j\neq i$ with $f_j(\x_i)=-1$, then 
$\x_i$ is in the relative interior of a facet of the 
polyhedral set $\LL$, moreover, by \eqref{eq:linearizationSC}, no other point 
of the set $\{\pm \x_1,\ldots,\pm\x_n\}$ lies on that facet.

If there is a $j\neq i$ with $f_j(\x_i)=-1$, then 
$\x_i$ is in the intersection of two facets of $\LL$, moreover, by 
\eqref{eq:linearizationSC}, no other point of the set $\{\pm 
\x_1,\ldots,\pm\x_n\}$ lies on the union of those two facets.

Thus, there is an $\oo$-symmetric, strictly convex body $\K\subset \LL$ which 
contains each $\x_i$. Clearly, for 
\iin, the hyperplane $\{\x\in\Ed\st f_i(\x)=1\}$ supports $\K$ at $\x_i$. It 
is an easy exercise to see that the family $2\x_1+\K,2\x_2+\K,\ldots,2\x_n+\K$ 
is a separable Hadwiger configuration of $\K$.

Next, assume that \xfn{} is a set of $n$ \pairs{} satisfying 
\eqref{eq:linearizationFirst} and 
\eqref{eq:linearizationSM}. To show that there is a smooth convex body 
$\K$ with $\Hsep(\K)\geq n$, one may either copy the above proof and make the 
obvious modifications, or use duality: interchange the role of the $\x_i$s with 
that of the $f_i$s, obtain a strictly convex body in the space of linear 
functionals, and then, by polarity obtain a smooth convex body in $\Ed$. We 
leave the details to the reader.

Finally, if \eqref{eq:linearizationFirst}, \eqref{eq:linearizationSC} and 
\eqref{eq:linearizationSM} hold, then in the above con\-struc\-ti\-on of a 
strictly 
convex body, we had that each point of the set $\{\pm 
\x_1,\ldots,\linebreak[0]\pm\x_n\}$ lies in the interior of a facet of $\LL$, 
with no other 
point lying on the same facet. Thus, there is an $\oo$-symmetric, smooth and 
strictly convex body $\K\subset \LL$ which contains each $\x_i$. Clearly, we 
have $\Hsep(\K)\geq n$.
\end{proof}

\begin{note}\label{note:projection}
Let $\K$ be an $\oo$-symmetric, strictly convex body in $\Ed$, and consider a 
separable Hadwiger configuration of $\K$ with $n$ members. Then, by 
Lemma~\ref{lem:linearization}, we have a set of $n$ vector-linear functional 
pairs satisfying \eqref{eq:linearizationFirst} and \eqref{eq:linearizationSC}.

If for each $\iin$, we have that $-\x_i$ is not in the set of vectors, then 
\eqref{eq:linearization} is automatically satisfied. We remark that in this 
case, we may replace $\K$ with a strictly convex \emph{and} smooth body.

If for some $k\neq \ell$ we have $\x_{\ell}=-\x_k$, then 
by \eqref{eq:linearizationFirst}, $f_j(\x_k)=0$ for all 
$j\in[n]\setminus\{k,\ell\}$. Thus, if we remove $(\x_k,f_k)$ and 
$(\x_{\ell},f_{\ell})$ from the set of \pairs, then we 
obtain $n-2$ pairs that still satisfy \eqref{eq:linearizationFirst} and 
\eqref{eq:linearizationSC}, and the linear functionals lie in a 
$(d-1)$-dimensional linear hyperplane. Thus, we may consider the problem of 
bounding their maximum number, $n-2$ in $\E^{d-1}$.

The same dimension reduction argument can be repeated when $\K$ is smooth. 
Thus, in order to bound $\Hsep(\K)$ for smooth \emph{ or } strictly convex 
bodies, it is sufficient to consider smooth \emph{ and } strictly convex 
bodies, and bound $n$ for which there are $n$ vectors with 
linear functionals satisfying \eqref{eq:linearization}.
\end{note}

We will rely on the following basic fact from convexity due to Steinitz 
\cite{St13} in its original form, and then refined later with 
the characterization of the case of equality, see \cite{Re65}.

\begin{lemma}\label{lem:steinitz}
Let $\xn$ be points in $\Ed$ with $\oo\in\inter\conv\{\xn\}$. Then there is a 
subset $A\subseteq\{\xn\}$ of cardinality at most $2d$ with $\oo\in\inter\conv 
A$. 

Furthermore, if the minimal cardinality of such $A$ is $2d$, then $A$ consists 
of the endpoints of $d$ line segments which span $\Ed$, and whose relative 
interiors intersect in $\oo$.
\end{lemma}

\begin{proposition}\label{prop:applysteinitzfirst}
Let \xfn{} be \pairs{} in $\Ed$ satisfying \eqref{eq:linearizationFirst}. 
Assume 
further that $\oo\in\inter\conv\{\xn\}$. Then $n\leq 2d$.

Moreover, if $n=2d$, then the points $\xn$ are vertices of a cross-polytope 
with center $\oo$.
\end{proposition}
\begin{proof}[Proof of Proposition~\ref{prop:applysteinitzfirst}]
By \eqref{eq:linearizationFirst}, for any proper subset $A\subsetneq\{\xn\}$, 
we 
have that the origin is not in the interior of $\conv A$. Thus, by 
Lemma~\ref{lem:steinitz}, $n\leq2d$. 

Next, assume that $n=2d$. Observe that it follows from 
\eqref{eq:linearizationFirst} that if $x_i=\lambda x_j$ for some \ijn{} and 
$\lambda\in\Re$, then $\lambda=-1$.
Thus, combining the argument in the previous paragraph with the second part of 
Lemma~\ref{lem:steinitz} yields the second part of 
Proposition~\ref{prop:applysteinitzfirst}.
\end{proof}

\begin{proposition}\label{prop:edge}
Let \xfn{} be \pairs{} in $\Ed$ satisfying \eqref{eq:linearization}. Assume 
that $\oo\notin\conv\{\xn\}$. Then for any $1\leq k<\ell \leq n$, the triangle 
$\conv\{\oo,\x_k,\x_\ell\}$ is a face of the convex polytope 
$\mathbf{P}:=\conv(\{\xn\}\cup\{\oo\})$.
\end{proposition}
\begin{proof}[Proof of Proposition~\ref{prop:edge}]
By \eqref{eq:linearization}, we have that $f_i(\x_j)>-1$ for all \ijn.
Suppose for a contradiction that $\conv\{\x_j\st j\in [n]\setminus\{k,\ell\}\}$ 
contains a point of the form $\x=\lambda \x_k+\mu \x_\ell$ with 
$\lambda,\mu\geq0, 
0<\lambda+\mu\leq1$. By \eqref{eq:linearization}, we 
have 
$f_k(\x), f_\ell(\x)\leq 0$. Thus,
\[
 0\geq f_k(\x)+f_\ell(\x)=
 \lambda(1+f_\ell(\x_k))+\mu(1+f_k(\x_\ell))>0,
\]
a contradiction.
\end{proof}

\section{Proofs of Theorems~\ref{thm:smoothstrictlycvx} and 
\ref{thm:strict}}\label{sec:thmproofs}

\begin{proof}[Proof of Theorem~\ref{thm:strict}]
To prove part (\ref{item:twosided}), we will use induction on $d$, the base 
case, $d=1$ being trivial.
By the dimension-reduction argument in Note~\ref{note:projection}, we may 
assume that there are $n$ \pairs{} \xfn{} satisfying \eqref{eq:linearization}. 

If $\oo\notin\conv\{\xn\}$, and $\oo\notin\conv\{\fn\}$, then, clearly, 
$n\leq\hsep(d)$.

Thus, we may assume that $\oo\in\conv\{\xn\}$.
We may also assume that $F=\conv\{\x_1,\ldots,\x_k\}$ is the face of the 
polytope 
$\conv\{\xn\}$ that \emph{supports} $\oo$, that is the face which contains 
$\oo$ in its relative interior.
Let $H:=\spann F$. If $H$ is the entire space $\Ed$, then 
$\oo\in\inter\conv\{\xn\}$ and hence, 
$n\leq2d$ follows from Proposition~\ref{prop:applysteinitzfirst}.

On the other hand, if $H$ is a proper linear subspace of $\Ed$, then clearly, 
for any $i>k$, we have that $f_i$ is identically zero on $H$.

Applying Proposition~\ref{prop:applysteinitzfirst} on $H$ for $\{\x_i\st i\leq 
k\}$ with $\{f_i|_{H}\st i\leq k\}$, we have 
\begin{equation}\label{eq:kleqtwodim}
k\leq 2\dim H. 
\end{equation}

Denote by $H^{\perp}$ the orthogonal 
complement of $H$, and by $P$ the orthogonal projection of $\Ed$ onto 
$H^{\perp}$. 
It is not hard to see that $P$ is one-to-one on the set $\{\x_i\st i>k\}$.
Moreover, the set of points $\{P\x_i\st i>k\}$, with 
linear functionals $\{f_i|_{H^{\perp}}\st i>k\}$ restricted to $H^{\perp}$, 
satisfy \eqref{eq:linearization} in $H^{\perp}$. 

Combining \eqref{eq:kleqtwodim} with the induction hypothesis applied on 
$H^\perp$, we complete the proof of part (\ref{item:twosided}).

For the three-dimensional bound in part (\ref{item:onesided}), suppose that 
$\oo\notin\linebreak[0]\conv\{\x_1,\linebreak[0]\ldots,\linebreak[0]\x_4\}
\in\E^3$. By Radon's 
lemma, the set 
$\{\oo,\linebreak[0]\x_1,\linebreak[0]\ldots,\linebreak[0]\x_4\}$ admits a 
partition 
into two parts whose convex hulls intersect contradicting 
Proposition~\ref{prop:edge}. The same proof yields the two and the 
four-dimensional 
statements, while the one-dimensional claim is trivial.

We use a projection argument to prove part (\ref{item:onesidedeuclidean}). 
Assume that $\x_1,\linebreak[0]\ldots,\linebreak[0]\x_n$ is a set of Euclidean 
unit vectors with 
$\iprod{\x_i}{\x_j}\in(-1,0]$ for all $\ijn$. Furthermore, let $\y$ be a unit 
vector with $\iprod{\y}{\x_i}>0$ for all \iin. Consider the set of vectors
$\x_i^{\prime}:=\x_i-\iprod{\y}{\x_i}\y$, $i=1,\ldots,n$, all lying in the 
hyperplane $\y^{\perp}$. Now, for \ijn, we have
\begin{equation*} 
\iprod{\x_i^{\prime}}{\x_j^{\prime}}=\iprod{\x_i}{\x_j}-\iprod{\y}{\x_i}\iprod{
\y}{\x_j}<0.
\end{equation*}
Thus, $\x_i^{\prime}, i=1,\ldots,n$ form a set of $n$ vectors in a 
$(d-1)$-dimensional space with pairwise obtuse angles. It is known \cite{DH, 
Ra, Ku}, 
or may be proved using the same projection argument and induction on the 
dimension (projecting orthogonally to $(\x_n^{\prime})^{\perp}$) that $n\leq d$ 
follows.
\end{proof}

\begin{example}\label{ex:fivedonesidedex}
By Lemma~\ref{lem:linearization}, it is sufficient to exhibit 6 vectors (with their
convex hull not containing $\oo$ in $\E^5$) and 
corresponding linear functionals satisfying \eqref{eq:linearization}.
Let the unit vectors $\ve_4,\ve_5,\ve_6$ be the vertices 
of an equilateral triangle centered at $\oo$ in the linear plane 
$\spann\{\ev_4,\ev_5\}$ of $\E^5$.
Let $\x_i=\ev_i$, for $i=1,2,3$, and 
let $\x_i=(\ev_1+\ev_2+\ev_3)/3 +\ve_i$, for $i=4,5,6$.
Observe that $\oo\notin\allowbreak\conv\{\x_1,\ldots,\x_6\}$, 
as 
$\iprod{\ev_1+\ev_2+\ev_3}{\x_i}>0$ for $i=1,\ldots,6$.

We define the following linear functionals.

$f_1(\x)=\linebreak[0]\iprod{\ev_1-\frac{\ev_2+\ev_3}{2}}{\x}$, 
$f_2(\x)=\linebreak[0]\iprod{\ev_2-\frac{\ev_1+\ev_3}{2}}{\x}$, 
$f_3(\x)\linebreak[0]=\linebreak[0]\iprod{\ev_3-\frac{\ev_1+\ev_2}{2}}{\x}$,
 and 
$f_i(\x)=\linebreak[0]\iprod{\ve_i}{\x}$, for $i=4,5,6$.
Clearly, \eqref{eq:linearization} holds.
\end{example}

\begin{proof}[Proof of Theorem~\ref{thm:smoothstrictlycvx}]
First, we prove part (\ref{item:threefourdim}).
If the origin is in the interior of the convex hull of the translation vectors, 
then Proposition~\ref{prop:applysteinitzfirst} yields $n\leq2d$ and the 
characterization of equality. In the case when $\oo\notin\inter\conv\{\x_i\}$, 
Theorem~\ref{thm:strict} combined with Note~\ref{note:projection}
yields $n<2d$.

The proof of part (\ref{item:DG}) follows closely a classical proof of Danzer 
and Gr\"unbaum on the maximum size of an antipodal set in $\E^d$ \cite{DG62}.

By Lemma~\ref{lem:linearization} and Note~\ref{note:projection}, we may assume 
that $\K$ is an $\oo$-symmetric smooth strictly convex body in $\E^d$. Assume 
that $2\x_1+\K,2\x_2+\K,\ldots,2\x_n+\K$ is a separable 
Hadwiger configuration of $\K$, where $\x_1,\dots ,\x_n\in\bd \K$.
Let $f_i$ denote the linear functional 
corresponding to the hyperplane that separates $\K$ from $2\x_i+\K$.

For each \iin, let $\K_i$ be the set that we obtain by applying a homothety of 
ratio $1/2$ with center $\x_i$ on the set $\K\cap\{\x\in\Ed\st f_i(\x)\geq 
0\}$, 
that is,
\[
\K_i:=\frac{1}{2}\left(\K\cap\{\x\in\Ed\st f_i(\x)\geq 
0\}\right)+\frac{\x_i}{2}.
\]
These sets are pairwise non-overlapping. 
In fact, it is easy to see that the following even stronger statement holds:  
\[
\left(\mu \x_i+{\rm int}\left(\frac{1}{2}\K\right)\right)\cap\left(\bigcup_{ 
j\neq i}\K_j\right)=\emptyset
\]
for any $\mu\geq 0$ and $1\leq i\leq n$.
On the other hand, $\vol{\K_i}=2^{-(d+1)}\vol{\K}$ by the 
central symmetry of $\K$, 
where $\vol\cdot$ stands for the $d$-dimensional volume of the 
given set.
We remark that -- unlike in the proof of the main 
result of \cite{DG62} by Danzer and Gr\"unbaum -- the sets $\K_i$ 
are not translates 
of each other. Since each $\K_i$ is contained in 
$\K\setminus\inter\left(\frac{1}{2}\K\right)$, we immediately 
obtain the bound $n\leq 2^{d+1}-2$.

To decrease the bound further, replace $\K_1$ by 
\[
\widehat{\K}_1:=\K\cap\{\x\in\Ed\st f_1(\x)\geq 
1/2\}.
\]
Now, $\widehat{\K}_1,\K_2,\ldots,\K_n$ are still pairwise non-overlapping, and
are contained in $\K\setminus\inter\left(\frac{1}{2}\K\right)$.
The smoothness of $\K$ yields $\widehat{\K}_1\supsetneq \K_1$, and hence, 
$\vol{\widehat{\K}_1}>2^{-(d+1)}\vol{\K}$.
This completes the proof of part (\ref{item:DG}) of 
Theorem~\ref{thm:smoothstrictlycvx}.

\end{proof}

\section{Proof of Theorem~\ref{thm:contactno}}\label{sec:contactproof}

We define a local version of a totally separable packing.

\begin{definition}
Let $\PP:=\{\x_i+\K\st i\in I\}$ be a finite or infinite packing of translates 
of $\K$, and $\rho>0$. We say that $\PP$ is \emph{$\rho$-separable} if for each 
$i\in I$ we have that the family $\{\x_j+\K\st j\in I, 
\x_j+\K\subset\x_i+\rho\K\}$
is a totally separable packing of translates of $\K$. Let $\dsep$ denote the 
largest density of a $\rho$-separable packing of translates of $\K$, that is,
\begin{equation*}
 \dsep:=\sup_{\PP}\limsup_{\lambda\rightarrow\infty}\frac{\sum\limits_{i: 
\x_i+\K\subset[-\lambda,\lambda]^d}\vol{\x_i+\K}}{(2\lambda)^d},
\end{equation*}
where the supremum is taken over all $\rho$-separable packings $\PP$ of 
translates of $\K$.
\end{definition}

We quote Lemma~1 of \cite{BeLa18}.

\begin{lemma}\label{lem:BeLa18}
 Let $\{\x_i+\K\st\iin\}$ be a $\rho$-separable packing of translates of an 
$\oo$-symmetric convex body $\K$ in $\Ed$ with $\rho\geq 1, n\geq 1$ and 
$d\geq2$. Then
\begin{equation*}
 \frac{n\vol{\K}}{\vol{\bigcup\limits_{\iin} \x_i+2\rho \K}}
\leq \dsep.
\end{equation*}

\end{lemma}

\begin{lemma}\label{lem:lsep}
 Let $\K$ be a smooth $\oo$-symmetric convex body in $\Ed$ with $\difour$. 
Then there is a $\lambda>0$ such that for any separable Hadwiger 
configuration $\{\K\}\cup\{\x_i+\K\st i=1,\ldots,2d\}$ of $\K$,
\begin{equation}\label{eq:lsep}
 \lambda \K\subseteq \bigcup_{i=1}^{2d} (\x_i+\lambda\K).
\end{equation}
holds. In particular, \eqref{eq:lsep} holds with $\lambda=2$ when $d=2$.
\end{lemma}

\begin{definition}
We denote the smallest $\lambda$ satisfying \eqref{eq:lsep} by $\lsep$, and 
note that $\lsep\linebreak[0]\geq 2$, since otherwise $\bigcup_{i=1}^{2d} 
(\x_i+\lambda\K)$ does not contain $\oo$. 
\end{definition}

\begin{proof}[Proof of Lemma~\ref{lem:lsep}]
Clearly, $\lambda$ satisfies \eqref{eq:lsep} if and only if, 
for each boundary point $\bb\in\bd(\K)$ we have that at least one 
of the $2d$ points $\bb-\frac{2}{\lambda}\x_i$ is in $\K$.

First, we fix a separable Hadwiger configuration of $\K$ with 
$2d$ members and show that for some $\lambda>0$, 
\eqref{eq:lsep} holds.
By Theorem~\ref{thm:smoothstrictlycvx}, we have that $\{\x_i\st 
i=1,\ldots,2d\}$ is an Auerbach basis of $\K$, and, in particular, the 
origin is in the interior of $\conv\{\x_i\st i=1,\ldots,2d\}$. It follows from 
the smoothness of $\K$ that for each boundary point 
$\bb\in\bd(\K)$ we have that 
at least one of the $2d$ rays $\{\bb-t\x_i\st t>0\}$ intersects the interior of 
$\K$. The existence of $\lambda$ now follows from the compactness of $\K$.

Next, since the set of Auerbach bases of $\K$ is compact (consider 
them as points in $\K^d$), it follows in a straightforward way that there is a 
$\lambda>0$, for which \eqref{eq:lsep} holds for all separable Hadwiger 
configurations of $\K$ with $2d$ members.

To prove the part concerning $d=2$, we make use of the characterization of 
the equality case in Part (\ref{item:threefourdim}) of 
Theorem~\ref{thm:smoothstrictlycvx}. An Auerbach basis of a planar 
$\oo$-symmetric convex body $\K$ means that $\K$ is contained in an 
$\oo$-symmetric parallelogram, the midpoints of whose edges are
$\pm\x_1,\pm\x_2$, and $\pm\x_1,\pm\x_2\in\K$. We leave it as an exercise to 
the reader that in this case, for each boundary point 
$\bb\in\bd(\K)$ we have that at least one of the $4$ 
points $\bb\pm\frac{\x_1}{2},\bb\pm\frac{\x_2}{2}$ is in $\K$.
\end{proof}

We denote the $(d-1)$-dimensional Hausdorff measure by $\vol[d-1]{\cdot}$, and 
the 
\emph{isoperimetric ratio} of a bounded set $\SSS\subset\Ed$ for 
which it is 
defined as
\[
 \Iq{\SSS}:=\frac{(\vol[d-1]{{\rm 
bd}\SSS})^d}{(\vol{\SSS})^{d-1}},
\]
and recall the \emph{isoperimetric inequality}, according to which it is 
minimized by Euclidean balls, that is, $\Iq{\B^d}\leq\Iq{\SSS}$ 
for any bounded 
set 
$\SSS\subset\Ed$, for which $\Iq{\SSS}$ is defined.

Finally, we are ready to state our main result, from which 
Theorem~\ref{thm:contactno} immediately follows.

\begin{theorem}\label{thm:contactnoPrecise}
 Let $\K$ be a smooth $\oo$-symmetric convex body in $\Ed$ with $\difour$. 
Then
 \begin{equation*}
  \csep\leq
 \end{equation*}
 \begin{equation*}
  dn-\frac{n^{(d-1)/d}}{2\left[\lsep\right]^{d-1}
  \left[\dsep[\frac{\lsep}{2},\K]\right]^{(d-1)/d}}
  \left[\frac{\Iq{\B^d}}{\Iq{\K}}\right]^{1/d}\leq
 \end{equation*}
 \begin{equation*}
  dn-\frac{n^{(d-1)/d} (\vol{\B^d})^{1/d}}{4\left[\lsep\right]^{d-1}}
 \end{equation*}
 for all $n>1$.
 
 In particular, in the plane, we have
 \begin{equation*}
  \csep\leq 2n-\frac{\sqrt{\pi}}{8}\sqrt{n}
  \end{equation*}
 for all $n>1$.
\end{theorem}

\begin{proof}[Proof of Theorem~\ref{thm:contactnoPrecise}]
Let $\PP=C+\K$ be a totally separable packing of translates of $\K$, where $C$ 
denotes the set of centers $C=\{\x_1,\ldots,\x_n\}$. Assume that $m$ of the $n$ 
translates is touched by the maximum number, that is, by 
Theorem~\ref{thm:smoothstrictlycvx}, $\Hsep(\K)=2d$ others. By 
Lemma~\ref{lem:lsep}, we have 
\begin{equation}\label{eq:lsepapplied}
\vol[d-1]{\bd\left( C+\lsep \K  \right)}\leq 
\end{equation}
\[
 (n-m)(\lsep)^{d-1}\vol[d-1]{\bd(\K)}.
\]

By the isoperimetric inequality, we have
\begin{equation}\label{eq:isoperimetric}
 \Iq{\B^d}\leq 
 \Iq{C+\lsep \K}=
 \frac{\left(\vol[d-1]{\bd\left( C+\lsep \K  
\right)}\right)^d}{\left(\vol{ C+\lsep \K}\right)^{d-1}}.
\end{equation}
Combining \eqref{eq:lsepapplied} and \eqref{eq:isoperimetric} yields
\[
 n-m\geq
 \frac{(\Iq{\B^d})^{1/d}\left[\vol{ C+\lsep 
\K}\right]^{(d-1)/d}}{(\lsep)^{d-1}\vol[d-1]{\bd \K}}.
\]
The latter, by Lemma~\ref{lem:BeLa18} is at least
\[
 \frac{(\Iq{\B^d})^{1/d}
 \left[
 \frac{n\vol{\K}}{\dsep[\lsep/2,\K]}
\right]^{(d-1)/d}}{(\lsep)^{d-1}\vol[d-1]{\bd \K}}.
\]
After rearrangement, we obtain the desired bound on $n$ completing the proof of 
the first inequality in Theorem~\ref{thm:contactnoPrecise}.

To prove the second inequality, we adopt the proof of \cite[Corollary~1]{Be02}.
First, note that $\dsep[\frac{\lsep}{2},\K]\leq 1$, and 
$(\Iq{\B^d})^{1/d}=d\vol{\B^d}$. Next, 
according to Ball's reverse isoperimetric inequality 
\cite{Ba91}, for any convex body $\K$, there is a non-degenerate affine map 
$T:\Ed\rightarrow\Ed$ with $\Iq{T\K}\leq(2d)^d$. Finally, notice that 
$\csep=\csep[T\K,n]$, and the inequality follows in a straightforward way.

The planar bound follows by substituting the value $\lsep=2$ from 
Lemma~\ref{lem:lsep}.
\end{proof}

\section{Remarks}\label{sec:remarks}

Lemma~\ref{lem:lsep} does not hold for strictly convex but not 
smooth convex bodies. Indeed, in $\E^3$, consider the $\oo$-symmetric polytope 
$\PPP:=\conv\{\pm \ev_1,\linebreak[0]\pm \ev_2,\linebreak[0]\pm 
\ev_3,\linebreak[0] \pm 0.9(\ev_1+\ev_2+\ev_3)\}$ where the 
$\ev_i$s are the standard 
basis vectors. The six translation vectors $\pm 
2\ev_1,\pm2\ev_2,\pm2\ev_3$ generate a 
separable Hadwiger configuration of $\PPP$. For the vertex 
$\bb:=0.9(\ev_1+\ev_2+\ev_3)$, we have that each of the 
$3$ lines $\{\bb+ t \ev_i\st 
t\in\Re\}$ intersect $\PPP$ in $\bb$ only. Thus, 
there is a strictly convex 
$\oo$-symmetric body $\K$ with the following properties. 
$\PPP\subset\K$, and $\pm 
\ev_i$ is a boundary point of $\K$ for each $i=1,2,3$, and at $\pm 
\ev_i$, the 
plane orthogonal to $\ev_i$ is a support plane of $\K$, and 
$\bb$ is a boundary 
point of $\K$, and the $3$ lines $\{\bb+ t \ev_i\st 
t\in\Re\}$ intersect $\K$ in $\bb$ only. For this strictly convex 
$\K$, we have 
$\lsep=\infty$.

Thus, it is natural to ask if in Theorem~\ref{thm:contactno} smoothness can be 
replaced by strict convexity. We note that in our proof, Lemma~\ref{lem:lsep} 
is 
the only place which does not carry over to this case.

The same construction of the polytope $\PPP$ shows that $\lsep$ may be 
arbitrarily 
large for a 
three-dimensional smooth convex body $\K$. Indeed, if we take 
$\K:=\PPP+\varepsilon\B^d$ with a small $\varepsilon>0$, we obtain a smooth 
body 
for which, by the previous argument, $\lsep$ is large.

Thus, it would be very interesting to see a lower bound on $f(\K)$ of 
Theorem~\ref{thm:contactno} which depends on $d$ only.


\section*{Acknowledgements}
K{\'a}roly Bezdek was partially supported by a Natural Sciences and 
Engineering Research Council of Canada Discovery Grant.
M{\'a}rton Nasz{\'o}di was partially supported by the National Research, 
Development and Innovation Office (NKFIH) grant NKFI-K119670 and by the J\'anos 
Bolyai Research Scholarship of the Hungarian Academy of 
Sciences, as well as the \'UNKP-17-4 New National Excellence Program of the 
Ministry of Human Capacities.


\end{document}